\documentclass[12pt]{amsart}
\usepackage{amsmath}
\usepackage{amsxtra}
\usepackage{amscd}
\usepackage{amsthm}
\usepackage{amsfonts}
\usepackage{amssymb}
\usepackage{eucal}
\usepackage{verbatim}
\usepackage[matrix,arrow,curve]{xy}

\newtheorem{lem}[subsection]{Lemma}
\newtheorem{prop}[subsection]{Proposition}

\newtheorem{conj}[subsection]{Conjecture}
\newtheorem{thm}[subsection]{Theorem}

\theoremstyle{definition}

\theoremstyle{remark}

\newcommand{\nc}{\newcommand}
\nc{\renc}{\renewcommand} \nc{\ssec}{\subsection}
\nc{\sssec}{\subsubsection}

\nc{\on}{\operatorname} \nc{\wh}{\widehat}
\nc\ol{\overline} \nc\ul{\underline} \nc\wt{\widetilde}

\nc{\BA}{{\mathbb{A}}} \nc{\BC}{{\mathbb{C}}}
\nc{\BD}{{\mathbb{D}}} \nc{\BG}{{\mathbb{G}}} \nc{\BQ}{{\mathbb{Q}}}
\nc{\BM}{{\mathbb{M}}} \nc{\BN}{{\mathbb{N}}}
\nc{\BP}{{\mathbb{P}}} \nc{\BR}{{\mathbb{R}}}
\nc{\BZ}{{\mathbb{Z}}} \nc{\BS}{{\mathbb{S}}} \nc{\BW}{{\mathbb{W}}}

\nc{\CA}{{\mathcal{A}}} \nc{\CB}{{\mathcal{B}}} \nc{\CalD}{{\mathcal{D}}}
\nc{\CE}{{\mathcal{E}}} \nc{\CF}{{\mathcal{F}}}
\nc{\CG}{{\mathcal{G}}} \nc{\CH}{{\mathcal{H}}}
\nc{\CI}{{\mathcal{I}}} \nc{\CK}{{\mathcal{K}}} \nc{\CL}{{\mathcal{L}}}
\nc{\CM}{{\mathcal{M}}} \nc{\CN}{{\mathcal{N}}}
\nc{\CO}{{\mathcal{O}}} \nc{\CP}{{\mathcal{P}}}
\nc{\CQ}{{\mathcal{Q}}} \nc{\CR}{{\mathcal{R}}}
\nc{\CS}{{\mathcal{S}}} \nc{\CT}{{\mathcal{T}}}
\nc{\CU}{{\mathcal{U}}} \nc{\CV}{{\mathcal{V}}}  \nc{\CY}{{\mathcal Y}}
\nc{\CW}{{\mathcal{W}}} \nc{\CZ}{{\mathcal{Z}}}

\nc{\cM}{{\check{\mathcal M}}{}} \nc{\csM}{{\check{\mathcal A}}{}}
\nc{\oM}{{\overset{\circ}{\mathcal M}}{}}
\nc{\obM}{{\overset{\circ}{\mathbf M}}{}}
\nc{\oCA}{{\overset{\circ}{\mathcal A}}{}}
\nc{\obA}{{\overset{\circ}{\mathbf A}}{}}
\nc{\ooM}{{\overset{\circ}{M}}{}}
\nc{\osM}{{\overset{\circ}{\mathsf M}}{}}
\nc{\vM}{{\overset{\bullet}{\mathcal M}}{}}
\nc{\nM}{{\underset{\bullet}{\mathcal M}}{}}
\nc{\oD}{{\overset{\circ}{\mathcal D}}{}}
\nc{\obD}{{\overset{\circ}{\mathbf D}}{}}
\nc{\oA}{{\overset{\circ}{\mathbb A}}{}}
\nc{\op}{{\overset{\bullet}{\mathbf p}}{}}
\nc{\cp}{{\overset{\circ}{\mathbf p}}{}}
\nc{\oU}{{\overset{\bullet}{\mathcal U}}{}}
\nc{\ofZ}{{\overset{\circ}{\mathfrak Z}}{}}

\nc{\ff}{{\mathfrak{f}}} \nc{\fv}{{\mathfrak{v}}}
\nc{\fa}{{\mathfrak{a}}} \nc{\fb}{{\mathfrak{b}}}
\nc{\fd}{{\mathfrak{d}}} \nc{\fe}{{\mathfrak{e}}}
\nc{\fg}{{\mathfrak{g}}} \nc{\fgl}{{\mathfrak{gl}}}
\nc{\fh}{{\mathfrak{h}}} \nc{\fri}{{\mathfrak{i}}}
\nc{\fj}{{\mathfrak{j}}} \nc{\fk}{{\mathfrak{k}}} \nc{\fl}{{\mathfrak{l}}}
\nc{\fm}{{\mathfrak{m}}} \nc{\fn}{{\mathfrak{n}}}
\nc{\ft}{{\mathfrak{t}}} \nc{\fu}{{\mathfrak{u}}}
\nc{\fw}{{\mathfrak{w}}} \nc{\fz}{{\mathfrak{z}}}
\nc{\fp}{{\mathfrak{p}}} \nc{\frr}{{\mathfrak{r}}}
\nc{\fs}{{\mathfrak{s}}} \nc{\fsl}{{\mathfrak{sl}}}
\nc{\hsl}{{\widehat{\mathfrak{sl}}}}
\nc{\hgl}{{\widehat{\mathfrak{gl}}}}
\nc{\hg}{{\widehat{\mathfrak{g}}}}
\nc{\chg}{{\widehat{\mathfrak{g}}}{}^\vee}
\nc{\hn}{{\widehat{\mathfrak{n}}}}
\nc{\chn}{{\widehat{\mathfrak{n}}}{}^\vee}

\nc{\fA}{{\mathfrak{A}}} \nc{\fB}{{\mathfrak{B}}} \nc{\fC}{{\mathfrak{C}}}
\nc{\fD}{{\mathfrak{D}}} \nc{\fE}{{\mathfrak{E}}}
\nc{\fF}{{\mathfrak{F}}} \nc{\fG}{{\mathfrak{G}}} \nc{\fH}{{\mathfrak{H}}}
\nc{\fI}{{\mathfrak{I}}} \nc{\fJ}{{\mathfrak{J}}}
\nc{\fK}{{\mathfrak{K}}} \nc{\fL}{{\mathfrak{L}}}
\nc{\fM}{{\mathfrak{M}}} \nc{\fN}{{\mathfrak{N}}}
\nc{\frP}{{\mathfrak{P}}} \nc{\fQ}{{\mathfrak{Q}}}
\nc{\fT}{{\mathfrak{T}}} \nc{\fU}{{\mathfrak{U}}}
\nc{\fV}{{\mathfrak{V}}} \nc{\fW}{{\mathfrak{W}}}
\nc{\fX}{{\mathfrak{X}}} \nc{\fY}{{\mathfrak{Y}}}
\nc{\fZ}{{\mathfrak{Z}}}

\nc{\ba}{{\mathbf{a}}}
\nc{\bb}{{\mathbf{b}}} \nc{\bc}{{\mathbf{c}}}
\nc{\be}{{\mathbf{e}}} \nc{\bj}{{\mathbf{j}}}
\nc{\bn}{{\mathbf{n}}} \nc{\bp}{{\mathbf{p}}}
\nc{\bq}{{\mathbf{q}}} \nc{\br}{{\mathbf{r}}} \nc{\bt}{{\mathbf{t}}}
\nc{\bfu}{{\mathbf{u}}} \nc{\bv}{{\mathbf{v}}}
\nc{\bx}{{\mathbf{x}}} \nc{\by}{{\mathbf{y}}}
\nc{\bw}{{\mathbf{w}}} \nc{\bA}{{\mathbf{A}}}
\nc{\bB}{{\mathbf{B}}} \nc{\bC}{{\mathbf{C}}}
\nc{\bD}{{\mathbf{D}}} \nc{\bF}{{\mathbf{F}}}
\nc{\bH}{{\mathbf{H}}} \nc{\bJ}{{\mathbf{J}}} \nc{\bK}{{\mathbf{K}}}
\nc{\bM}{{\mathbf{M}}} \nc{\bN}{{\mathbf{N}}}
\nc{\bO}{{\mathbf{O}}} \nc{\bS}{{\mathbf{S}}} \nc{\bT}{{\mathbf{T}}}
\nc{\bV}{{\mathbf{V}}} \nc{\bW}{{\mathbf{W}}}
\nc{\bX}{{\mathbf{X}}}
\nc{\bY}{{\mathbf{Y}}} \nc{\bP}{{\mathbf{P}}}
\nc{\bZ}{{\mathbf{Z}}} \nc{\bh}{{\mathbf{h}}}

\nc{\sA}{{\mathsf{A}}} \nc{\sB}{{\mathsf{B}}}
\nc{\sC}{{\mathsf{C}}} \nc{\sD}{{\mathsf{D}}}
\nc{\sE}{{\mathsf{E}}} \nc{\sF}{{\mathsf{F}}} \nc{\sG}{{\mathsf{G}}}
\nc{\sI}{{\mathsf{I}}} \nc{\sK}{{\mathsf{K}}} \nc{\sL}{{\mathsf{L}}}
\nc{\sfm}{{\mathsf{m}}} \nc{\sM}{{\mathsf{M}}} \nc{\sO}{{\mathsf{O}}}
\nc{\sQ}{{\mathsf{Q}}} \nc{\sP}{{\mathsf{P}}}
\nc{\sT}{{\mathsf{T}}} \nc{\sZ}{{\mathsf{Z}}}
\nc{\sV}{{\mathsf{V}}} \nc{\sW}{{\mathsf{W}}}
\nc{\sfp}{{\mathsf{p}}} \nc{\sr}{{\mathsf{r}}}
\nc{\st}{{\mathsf{t}}} \nc{\sfb}{{\mathsf{b}}}
\nc{\sfc}{{\mathsf{c}}} \nc{\sd}{{\mathsf{d}}}
\nc{\sz}{{\mathsf{z}}}

\nc{\tA}{{\widetilde{\mathbf{A}}}}
\nc{\tB}{{\widetilde{\mathcal{B}}}}
\nc{\tg}{{\widetilde{\mathfrak{g}}}} \nc{\tG}{{\widetilde{G}}}
\nc{\TM}{{\widetilde{\mathbb{M}}}{}}
\nc{\tO}{{\widetilde{\mathsf{O}}}{}}
\nc{\tU}{{\widetilde{\mathfrak{U}}}{}} \nc{\TZ}{{\tilde{Z}}}
\nc{\tx}{{\tilde{x}}} \nc{\tbv}{{\tilde{\bv}}}
\nc{\tfP}{{\widetilde{\mathfrak{P}}}{}} \nc{\tz}{{\tilde{\zeta}}}
\nc{\tmu}{{\tilde{\mu}}}

\nc{\urho}{\underline{\rho}} \nc{\uB}{\underline{B}}
\nc{\uC}{{\underline{\mathbb{C}}}} \nc{\ui}{\underline{i}}
\nc{\uj}{\underline{j}} \nc{\ofP}{{\overline{\mathfrak{P}}}}
\nc{\oB}{{\overline{\mathcal{B}}}}
\nc{\og}{{\overline{\mathfrak{g}}}} \nc{\oI}{{\overline{I}}}

\nc{\eps}{\varepsilon} \nc{\hrho}{{\hat{\rho}}}
\nc{\blambda}{{\bar\lambda}} \nc{\bmu}{{\bar\mu}} \nc{\bnu}{{\bar\nu}}

\nc{\one}{{\mathbf{1}}} \nc{\two}{{\mathbf{t}}}

\nc{\Sym}{{\mathop{\operatorname{\rm Sym}}}}
\nc{\Tot}{{\mathop{\operatorname{\rm Tot}}}}
\nc{\Spec}{{\mathop{\operatorname{\rm Spec}}}}
\nc{\Ker}{{\mathop{\operatorname{\rm Ker}}}}
\nc{\Hilb}{{\mathop{\operatorname{\rm Hilb}}}}
\nc{\End}{{\mathop{\operatorname{\rm End}}}}
\nc{\Ext}{{\mathop{\operatorname{\rm Ext}}}}
\nc{\Hom}{{\mathop{\operatorname{\rm Hom}}}}
\nc{\CHom}{{\mathop{\operatorname{{\mathcal{H}}\it om}}}}
\nc{\GL}{{\mathop{\operatorname{\rm GL}}}}
\nc{\gr}{{\mathop{\operatorname{\rm gr}}}}
\nc{\Id}{{\mathop{\operatorname{\rm Id}}}}
\nc{\defi}{{\mathop{\operatorname{\rm def}}}}
\nc{\length}{{\mathop{\operatorname{\rm length}}}}
\nc{\supp}{{\mathop{\operatorname{\rm supp}}}}
\nc{\HC}{{\mathcal H}{\mathcal C}}

\nc{\Cliff}{{\mathsf{Cliff}}}
\nc{\Fl}{{\mathcal{F}\ell}} \nc{\Fib}{{\mathsf{Fib}}}
\nc{\Coh}{{\mathsf{Coh}}} \nc{\FCoh}{{\mathsf{FCoh}}}

\nc{\reg}{{\text{\rm reg}}}
\nc{\gvee}{{\mathfrak g}^{\!\scriptscriptstyle\vee}}
\nc{\tvee}{{\mathfrak t}^{\!\scriptscriptstyle\vee}}
\nc{\nvee}{{\mathfrak n}^{\!\scriptscriptstyle\vee}}
\nc{\bvee}{{\mathfrak b}^{\!\scriptscriptstyle\vee}}
       \nc{\rhovee}{\rho^{\!\scriptscriptstyle\vee}}

\nc{\cplus}{{\mathbf{C}_+}} \nc{\cminus}{{\mathbf{C}_-}}
\nc{\cthree}{{\mathbf{C}_*}} \nc{\Qbar}{{\bar{Q}}}

\nc{\Gtimes}{\vphantom{j^{X^2}}\smash{\overset{G}{\vphantom{\rule{0pt}{0.30em}}\smash{\times}}}}
\nc{\sGtimes}{\vphantom{j^{X^2}}\smash{\overset{\mathsf G}{\vphantom{\rule{0pt}{0.30em}}\smash{\times}}}}

\nc{\bOmega}{{\overline{\Omega}}}

\nc{\seq}[1]{\stackrel{#1}{\sim}}

\nc{\aff}{{\operatorname{aff}}}
\nc{\fin}{{\operatorname{fin}}}
\nc{\Gr}{{\operatorname{Gr}}}
\nc{\GR}{{\mathbf{Gr}}}
\nc{\Perv}{{\operatorname{Perv}}}
\nc{\Rep}{{\operatorname{Rep}}}
\nc{\IC}{{\operatorname{IC}}}
\nc{\Bun}{{\operatorname{Bun}}}
\nc{\Proj}{{\operatorname{Proj}}}
\nc{\pt}{{\operatorname{pt}}}
\nc{\bfmu}{{\boldsymbol{\mu}}}
\nc{\bfomega}{{\boldsymbol{\omega}}}
\nc{\calM}{\mathcal M}
\nc{\calA}{\mathcal A}
\nc{\calO}{\mathcal O}
\nc{\CC}{\mathbb C}
\nc{\calN}{\mathcal N}
\nc{\grg}{\mathfrak g}
\nc{\tslash}{/\!\!/\!\!/}
\nc\grt{\mathfrak t}
\nc\bfM{\mathbf M}
\nc\bfN{\mathbf N}
\nc\Sig{\Sigma}
\nc\ZZ{\mathbb{Z}}
\nc\calC{\mathcal C}
\nc\calF{\mathcal F}
\nc\calX{\mathcal X}
\nc\calY{\mathcal Y}
\nc\QCoh{\operatorname{QCoh}}
\nc\IndCoh{\operatorname{IndCoh}}
\nc\Maps{\operatorname{Maps}}
\nc\Dmod{D\operatorname{-mod}}

\nc{\calD}{\mathcal D}
\nc\bfO{\mathbf O}
\renewcommand{\AA}{\mathbb A}
\nc\GG{\mathbb G}
\nc\calK{\mathcal K}
\nc{\calG}{\mathcal G}
\nc\RHom{\operatorname{RHom}}
\nc\grs{\mathfrak s}
\nc{\tilX}{\widetilde X}
\nc\calB{\mathcal B}
\nc\calS{\mathcal S}
\nc\calT{\mathcal T}
\nc\calZ{\mathcal Z}
\nc\LS{\operatorname{LocSys}}
\nc\bfL{\on{\mathbf L}}
\nc\calE{\on{\mathcal E}}
\nc\calW{\on{\mathcal W}}
\nc\kap{\kappa}
\nc\x{\times}
\nc\calL{\mathcal L}
\nc\Lam{\Lambda}
\nc\lam{\lambda}
\nc\SL{\on{SL}}
\nc\oGr{\overline{\Gr}}
\nc\bfc{\mathbf c}
\nc\bfa{\mathbf a}
\nc\PGL{\on{PGL}}
\nc\VV{\mathbb V}
\nc\tilF{\widetilde{F}}
\nc\triv{\on{triv}}
\renewcommand\Coh{\operatorname{Coh}}
\newcommand{\dbkts}[1]{[\![#1]\!]}
\newcommand{\dprts}[1]{(\!(#1)\!)}
\nc\bfH{\mathbf H}
\renewcommand\SS{\mathbb S}
\begin{document}
\author[A.~Braverman]{Alexander Braverman}
\address{
Department of Mathematics,
University of Toronto and Perimeter Institute of Theoretical Physics,
Waterloo, Ontario, Canada, N2L 2Y5
\newline
Skolkovo Institute of Science and Technology;
}
\email{braval@math.toronto.edu}
\author[M.~Finkelberg]{Michael Finkelberg}
\address{
National Research University Higher School of Economics, Russian Federation\newline
Department of Mathematics, 6 Usacheva st., Moscow 119048;\newline
Skolkovo Institute of Science and Technology;\newline
Institute for Information Transmission Problems}
\email{fnklberg@gmail.com}

\title
{A quasi-coherent description of the category $\Dmod(\Gr_{\GL(n)})$}

\dedicatory{To Sasha Beilinson and Vitya Ginzburg}




\begin{abstract}
In \cite{cime} we have formulated a conjecture describing the derived category $\Dmod(\Gr_{\GL(n)})$ of (all) $D$-modules on the affine Grassmannian of the group $\GL(n)$ as the category of quasi-coherent sheaves on a certain stack (it is explained in {\em loc.\ cit.} that this conjecture ``follows" naturally from some heuristic arguments involving 3-dimensional quantum field theory). In this paper we prove a weaker version of this conjecture for the case $n=2$.
\end{abstract}
\maketitle

\section{Introduction and statement of the results}
\subsection{General notation}
In general we work over $\CC$.

For a (derived) stack $\calY$ we denote by $\QCoh(\calY)$ the derived category of quasi-coherent sheaves on $\calY$ and by $\Dmod(\calY)$ the derived category of $D$-modules on $\calY$. In addition we are going to denote by $\IndCoh(\calY)$ the derived category of ind-coherent sheaves on $\calY$;
this category coincides with $\QCoh(\calY)$ when $\calY$ is a classical (non-derived) smooth stack but in general the two are different (we are going to use \cite{ArGa} as our main reference for the notion and properties of ind-coherent sheaves).

Let $\calO=\CC\dbkts{z},\calK=\CC\dprts{z}$. Set $\calD=\Spec(\calO),\calD^*=\Spec(\calK)$.
By a local system of rank $n$ on $\calD^*$ we shall mean a vector bundle $\calE$ on $\calD^*$ of rank $n$
endowed with a connection $\nabla\colon \calE\to \calE\otimes \Omega^1_{\calD^*}$.
We denote by $\LS_n(\calD^*)$ the stack of local systems of rank $n$ on $\calD^*$.

For an algebraic group $G$ over $\CC$ we denote by $\Gr_G=G(\calK)/G(\calO)$ the affine Grassmannian of $G$ (viewed as an ind-scheme).

\subsection{The main conjecture: $\GL(n)$-case}

Let $\calW_n$ denote the stack which classifies the following data:

\medskip
(1) A local system $\calE_i$ on $\calD^*$ of rank $i$ for any $i=1,\ldots,n$.

(2) A morphism $\kap_i\colon \calE_i\to\calE_{i+1}$ for any $i<n$.

\medskip
\noindent
This stack maps naturally to the stack $\LS_n(\calD^*))$ (this map sends the above data to $\calE_n$). The trivial local system defines a map
$\pt/\GL(n)\to \calW_n$ and we let $\calW_n^{\triv}$ denote the cartesian product
$$
\begin{CD}
\calW_n^{\triv} @>>> \calW_n\\
@VVV @VVV\\
\pt/\GL(n)@>>> \LS_n(\calD^*).
\end{CD}
$$
It is worthwhile to note that $\calW_n^{\triv}$ is a dg-stack.

The following is a slightly corrected version of a conjecture formulated in \cite{cime}:
\begin{conj}\label{main-gln}
The category $\IndCoh(\calW_n^{\triv})$ is equivalent to the category $\Dmod(\Gr_{\GL(n)})$. This equivalence respects the natural action of the tensor category $\Rep(\GL(n))$ on both sides.
\end{conj}
It is explained in \cite{cime} how to ``deduce" Conjecture \ref{main-gln} from quantum field theory considerations. In this paper we are not going to discuss this physical motivation at all: instead we are going to present some mathematical evidence for it (mostly in the case $n=2$).

\subsection{The main conjecture: $\GL(2)$-case}
In this subsection we would like to strengthen Conjecture \ref{main-gln} in the case of $\GL(2)$. First, let us ask a natural question for arbitrary $n$. Namely, it is clear that the category $\IndCoh(\calW_n^{\triv})$ lives over $\prod_{i=1}^{n-1} \LS_i(\calD^*)$. How to see this structure on $\Dmod(\Gr_{\GL(n)})$?

We don't know the answer to this question except for the case $n=2$. To explain the answer we need to recall the statement of geometric local class field theory (due to G.~Laumon, cf.~\cite{laum}):

\begin{thm}\label{cft}
  There is a natural equivalence of monoidal categories
  $\Dmod(\calK^{\x})\simeq \QCoh(\LS_1(\calD^*))$.\footnote{In this case the equivalence actually holds on the level of abelian categories, but the equivalence of Conjecture \ref{main-gln} only has a chance to hold on the derived level. Also in this case there is no difference between QCoh and IndCoh.}
\end{thm}
Theorem \ref{cft} implies that the structure of ``living over $\LS_1(\calD^*)$" on a
category $\calC$ is the same as a strong action of $\calK^{\x}$ on $\calC$ (see e.g.~\cite[4.1.2]{Ga}).
Thus to answer our question for $n=2$ it is enough to describe a strong action of $\calK^{\x}$ on the category $\Dmod(\Gr_{\GL(n)})$.
Since the group $\GL(2,\calK)$ acts strongly on $\Dmod(\Gr_{\GL(2)})$, it is enough to describe a map $\calK^{\x}\to \GL(2,\calK)$. The relevant map is given by
$$
x\overset{\eta}\mapsto
\begin{pmatrix}
x & 0\\
0 & 1
\end{pmatrix}
$$

So, we get the following conjecture:
\begin{conj}\label{main-gl2}
The category $\IndCoh(\calW_2^{\triv})$ is equivalent to the category $\Dmod(\Gr_{\GL(2)})$. This equivalence respects the natural action of the tensor category $\Rep(\GL(2))$ on both sides. In addition the action of the tensor category $\QCoh(\LS_1(\calD^*))\simeq \Dmod(\calK^\x)$ on $\Dmod(\Gr_{\GL(2)})$  coming from the natural projection $\calW_2^{\triv}/\GL(2)\to \LS_1(\calD^*)$ under the above equivalence corresponds to the action of $\Dmod(\calK^\x)$ coming
from the embedding $\eta\colon \calK^\x\to \GL(2,\calK)$ defined above.
\end{conj}

\subsection{Fiberwise version}
We don't know how to prove Conjecture \ref{main-gl2} either. The purpose of this paper is to prove a weaker statement: namely we are going to show that the fibers of both $\IndCoh(\calW_2^{\triv})$ and of $\Dmod(\Gr_{\GL(2)})$ over any $\calE\in \LS_1(\calD^*)$ are equivalent.
Let us look at these fibers in more detail.

Denote by $\pi$ the natural projection $\calW_2^{\triv}\to \LS_1(\calD^*)$. Let $\calE\in \LS_1(\calD^*)$.  Let us first work with QCoh instead of IndCoh. Then the fiber of
$\QCoh(\calW_2^{\triv})$ over $\calE$ (which we shall denote by $\QCoh(\calW_2^{\triv})_{\calE}$)
is equivalent to $\QCoh(\pi^{-1}(\calE))$.\footnote{Here $\pi^{-1}(\calE)$ should be understood in dg-sense.}
Assume that $\calE$ is non-trivial. Then any morphism from $\calE$ to the trivial local system of rank 2 is 0; in other words away from the trivial local system (of rank 1) the natural map $\calW_2^{\triv}\to \LS_1(\calD^*)\times \pt/\GL(2)$ is an isomorphism. Hence $\pi^{-1}(\calE)=\pt/\GL(2)$ and in this case $\QCoh(\calW_2^{\triv})_{\calE}$ is equivalent to $\Rep(\GL(2))$.

On the other hand, assume that $\calE$ is trivial. Then $\pi^{-1}(\calE)$ is a dg-stack equivalent to $(\VV\x \VV[-1])/\GL(\VV)$ where $\VV$ is a two-dimensional vector space over $\CC$ (this follows from
the fact the the dg-scheme classifying $f\in \bfO_{\calD^*}$ such that $df=0$ is $\AA^1\times \AA^1[-1]$).

Let us go back to the IndCoh story. Assume that we have a morphism $\pi\colon \calY\to \calX$ of (dg) stacks; assume moreover that $\calX$ is a smooth classical stack.  In this case the fiber of $\IndCoh(\calY)$ over a point $x\in \calX$ is described in Section 2 of \cite{ArGa}. We are not going to reproduce that general answer here as it will require introducing more cumbersome notation; let us just explain what this answer amounts to in the case when $\calY=\calW_2^{\triv}$ and $\calX=\LS_1(\calD^*)$.

Let $\calE$ be a rank 1 local system on $\calD^*$ as above. First, if $\calE$ is non-trivial, then it is easy to see that the fiber $\IndCoh(\calW_2^{\triv})_{\calE}$ of $\IndCoh(\calW_2^{\triv})$ over $\calE$ is just $\Rep(\GL(2))$ as before. Let now $\calE$ be trivial.
Then, as was noted above we have the isomorphism
$$
\pi^{-1}(\calE)\simeq (\VV\times \VV[-1])/\GL(\VV),
$$
where $\VV$ is a two-dimensional vector space. By Koszul duality, the category $\IndCoh((\VV\times \VV[-1])/\GL(\VV))$ is equivalent to the derived category of $\GL(\VV)$-equivariant dg-modules over $\bfO_{\VV\times \VV^*[2]}$. On the other hand, the sought-for fiber $\IndCoh(\calW_2^{\triv})_{\calE}$ is equivalent to the derived category of $\GL(\VV)$-equivariant dg-modules over $\bfO_{\VV\times \VV^*[2]}$ which are set-theoretically supported on $\calZ_{\VV}\subset \VV\times \VV^*[2]$ consisting of pairs $(v,v^*)$ with $v^*(v)=0$. We shall denote this category by $\IndCoh_{\calZ_{\VV}}((\VV\times \VV[-1])/\GL(\VV))$.

Now, any $\calE$ as above defines a character $D$-module $\calL$ on $\calK^\x$, i.e.\ a rank 1 local system endowed with an isomoprhism
$m^*\calL\simeq \calL\boxtimes \calL$ (here $m\colon \calK^\x\times \calK^\x\to\calK^\x$ is the multiplication map) satisfying the standard associativity constraint. Under this correspondence trivial $\calE$ corresponds to trivial $\calL$, i.e.\ $\calL$ isomorphic to $\bfO_{\calK^\x}$ (note that $\calL$ is trivial if and only if it is trivial when restricted to $\calO^{\x}$). Given any $\calL$ as above, and a category $\calC$ with a strong action of $\calK^{\x}$ it makes sense to consider the category of $(\calK^{\x},\calL)$-equivariant objects in $\calC$. When $\calL$ is trivial, this is just the category of $\calK^{\x}$-equivariant objects.

Thus the following result is exactly the ``fiberwise version" of Conjecture \ref{main-gl2}:

\begin{thm}\label{thm}
Let $\calK^{\x}$ act on $\Gr_{\GL(2)}$ by means of the map $\eta$. Then
\begin{enumerate}
\item
Let $\calL$ be a non-trivial character $D$-module on $\calK^{\x}$. Then the category of $(\calK^{\x},\calL)$-equivariant $D$-modules on
$\Gr_{\GL(2)}$ is equivalent to $\Rep(\GL(2))$.
\item
Let $D^b_{\calK^{\x}}(\Gr_{\GL(2)})$ denote the full subcategory of the derived category of $\calK^{\x}$-equivariant $D$-modules on
$\Gr_{\GL(2)}$ whose restriction to any connected component of $\Gr_{\GL(2)}$ is a bounded complex whose cohomology $D$-modules have finite-dimensional support and are coherent. Then $D^b_{\calK^{\x}}(\Gr_{\GL(2)})$ is equivalent to $\Coh( (\VV\x \VV[-1])/\GL(\VV))$ (here $\VV$ is again a two-dimensional vector space over $\CC$).
\item
Let $D_{\calK^{\x}}(\Gr_{\GL(2)})$ denote the derived category of $\calK^{\x}$-equivariant $D$-modules on $\Gr_{\GL(2)}$. Then an object of $D_{\calK^{\x}}(\Gr_{\GL(2)})$ is compact if and only if

(a) It lies in $D^b_{\calK^{\x}}(\Gr_{\GL(2)})$;

(b) Its image under the equivalence (2) lies in $\Coh_{\calZ_{\VV}}((\VV\times \VV[-1])/\GL(\VV))$.

\noindent
In particular, the equivalence (ii) extends to the equivalence between $D_{\calK^{\x}}(\Gr_{\GL(2)})$ and $\IndCoh_{\calZ_{\VV}}((\VV\times \VV[-1])/\GL(\VV))$.
\end{enumerate}
\end{thm}

The rest of the paper is devoted to the proof of Theorem \ref{thm}.

\medskip
\noindent
{\bf Remarks.}
The fact that usually not all objects of the bounded equivariant derived category of $D$-modules (or constructible sheaves) are compact was first observed and studied by V.~Drinfeld and D.~Gaitsgory,
cf.~\cite{DrGa}. Also the reader should compare the last two assertions of Theorem \ref{thm} with respectively Theorem 12.3.3 and Corollary 12.5.5 of \cite{ArGa}.

\subsection{Acknowledgements}
This paper resulted from numerous conversations of the first-named author with D.~Gaitsgory and S.~Raskin which took place
during the workshop ``Vertex algebras, factorization algebras and applications" at IPMU in July 2018. The authors thank both D.~Gaitsgory and S.~Raskin
for their patient explanations and the organizers of the workshop for hospitality and for providing this opprtunity. We would also like to thank Roman Bezrukavnikov for help with some technical details of the proof.
M.F.\ was partially funded within the framework of the HSE University Basic Research Program and
the Russian Academic Excellence Project `5-100'.
\section{Proof of Theorem \ref{thm}(1)}
\subsection{Sketch of the proof}
In what follows we denote by $\Lam=\ZZ\oplus\ZZ$ the coweight lattice of $\GL(2)$ and by
$$
\Lam^+=\{ (a,b)\in \Lam|\ a\geq b\}
$$
the cone of dominant coweights.
Fix now a non-trivial character $D$-module $\calL$ on $\calK^{\x}$.
We claim that in order to prove Theorem \ref{thm}(1) it is enough to construct an embedding $\iota_{\calL}$ from $\Lam^+$ into the set of
$\calK^{\x}$-orbits on $\Gr_{\GL(2)}$ such that the following 3 properties hold:

(i) A $\calK^{\x}$-orbit on $\Gr_{\GL(2)}$ supports a $(\calK^{\x},\calL)$-equivariant $D$-module if and only if it lies in the image of $\iota_{\calL}$.

\noindent
In what follows for every $\lam\in \Lam^+$ let us denote by $\calF^{\lam}_!$ and $\calF^{\lam}_*$ the $!$ and $*$-extensions to all of $\Gr_{\GL(2)}$ of the corresponding irreducible $(\calK^{\x},\calL)$-equivariant $D$-module on the orbit $\iota_{\calL}(\lam)$.

(ii) For any $\lam\in \Lam^+$ we have
$$
\calF^0_!\star \IC^{\lam}\simeq \calF^{\lam}_!;\quad \calF^0_*\star \IC^{\lam}\simeq \calF^{\lam}_*.
$$

(iii) The natural morphism $\calF^0_!\to\calF^0_*$ is an isomorphism.

\medskip
Indeed, (ii) and (iii) together imply that the map $\calF^{\lam}_!\to \calF^{\lam}_*$ is an isomorphism for any $\lam$. Hence the category of $(\calK^{\x},\calL)$-equivariant $D$-modules on $\Gr_{\GL(2)}$ is semi-simple with simple objects $\calF^{\lam}:=\calF^{\lam}_!\simeq \calF^{\lam}_*$. Now (ii) implies that the functor $\calS\mapsto \calF^0\star \calS$ from $\Dmod_{\GL(2,\calO)}(\Gr_{\GL(2)})$ to the (abelian) category of $(\calK^{\x},\calL)$-equivariant $D$-modules on $\Gr_{\GL(2)}$ is an equivalence which is exactly what we had to prove.

So, it remains to define the map $\iota_{\calL}$ and to check the properties (i)--(iii).

\subsection{The map $\iota_{\calL}$}
There exists unique $k>0$ such that $\calL$ is pulled back from $\calO^{\x}/1+z^k\calO$ but not pulled back from $\calO^{\x}/1+z^{k-1}\calO$.
The corresponding map $\iota_{\calL}$ will only depend on $k$ which will be fixed till the end of this Section.
To simplify the notation we shall simply write $Y_{\lam}$ for the $\calK^{\x}$-orbit of $z^{\iota_{\calL}(\lambda)}$. Also we set
$X_{\lambda}$ to be the intersection of $Y_{\lambda}$ with $\Gr_{\SL(2)}$.

Let $\lam=(n_1,n_2)$ with $n_1\geq n_2$.
Then we set $Y_{\lam}$ to be the $\calK^{\x}$-orbit of the (image in $\Gr_{\GL(2)}$ of the) matrix

$$
\begin{pmatrix}
1 & z^{-k-n}\\
0 & 1
\end{pmatrix} \cdot
\begin{pmatrix}
z^{-n_2} & 0\\
0 & z^{n_2}
\end{pmatrix}
$$
Here $n=n_1+n_2$.

\subsection{Proof of (i)}\label{i}
It is enough to deal with $\calO^{\x}$-orbits on $\Gr_{\SL(2)}$ instead of $\calK^{\x}$-orbits on $\Gr_{\GL(2)}$.
Such orbits are in one-to-one correspondence with pairs $(m,l)\in \ZZ\x\ZZ$ with $l-2m\leq 0$; the $\calO^{\x}$-orbit corresponding to a given $(m,l)$
is the orbit of the matrix
$$
\begin{pmatrix}
1 & z^l\\
0 & 1
\end{pmatrix} \cdot
\begin{pmatrix}
z^m & 0\\
0 & z^{-m}
\end{pmatrix}
$$
The stabilzer of the above point in $\calO^{\x}$  is $1+z^{2m-l}\calO$. Hence this orbit supports a $(\calO^{\x},\calL)$-equivariant $D$-module if and only if $2m-l\geq k$. This is exactly the condition that there exists a pair $(n_1,n_2)\in \ZZ\times \ZZ$ such that $n_1\geq n_2$ satisfying
the equations
$$
l=-k-n,\quad m=-n_2.
$$

\subsection{Proof of (ii)}
Let us compute the convolution of $\calF^0_!$ with $\IC^{\lambda}$ where $\lambda=(n_1,n_2)$ (the corresponding calculation for $\calF^0_*$ is completely analogous). We need to show the following two things:

(1) The $*$-restriction $\calF^0_!\star\IC^{\lambda}$ to $X_{\lambda}$ is equal to IC-sheaf of $X_{\lambda}$;

(2) The $*$-restriction $\calF^0_!\star\IC^{\lambda}$ to any $\calO^*$-orbit on $\Gr_{\SL(2)}$ different from  $X_\lambda$ is equal to $0$.

\medskip
\noindent
For this it is enough to compute the stalk of $\calF^0_!\star\IC^{\lambda}$ at any point of the form
$$
g=
\begin{pmatrix}
1 & z^l\\
0 & 1
\end{pmatrix} \cdot
\begin{pmatrix}
z^m & 0\\
0 & z^{-m}.
\end{pmatrix}
$$
Let us fix $\lambda=(n_1,n_2), m,l$ and $k$ and let
$$
Z=\{ x\in X_0|\ x^{-1}g\in \oGr_{\GL(2)}^{\lambda}\}.
$$
Let $i$ denote the natural map from $Z$ to $X_0\simeq \calO^*/1+z^k\calO$. Then the above stalk is equal to $H^*_c(Z, i^*\calL[\dim X_0+\dim \Gr_{\GL(2)}^{\lambda}])$. We can assume that $x$ is of the form
$$
x=
\begin{pmatrix}
z^{-n} & az^{-n-k}\\
0 & 1
\end{pmatrix}
$$
where $a\in \calO^{\x}$. Then
$$
x^{-1}g=
\begin{pmatrix}
z^{n+m} & z^{n+l-m}-az^{-k-m}\\
0 & z^{-m}
\end{pmatrix}.
$$
This matrix defines a point in $\oGr_{\GL(2)}^{\lambda}$ if $n+m,-m\geq n_2$ and
$z^{-m}(z^{n+l}-az^{-k})\in z^{n_2}\calO$. Let $a=\sum a_i z_i$. We see that if $-m>n_2$ then
changing $a_{k-1}$ does not affect the above conditions; so, ``integrating out" $a_{k-1}$ first
we see that $H^*_c(Z, i^*\calL[\dim X_0+\dim \Gr_{\GL(2)}^{\lambda}])=0$. Assume now that $-m=n_2$. Then unless
$n+l=-k$ the above equations have no solutions, hence the sought-for stalk is again 0. The case $-m=n_2, n+l=-k$
is precisely the case $g\in X_{\lambda}$. In this case we must have $a_0=1$ and $ a_j=0$ for $0<j<k$.
So $Z$ consists of just one point and
$H^*_c(Z,\CC[ \dim X_0+\dim \Gr_{\GL(2)}^{\lambda}]=\CC[\dim X_{\lambda}]$ (since it is easy to see that $\dim X_0+\dim \Gr_{\GL(2)}^{\lambda}=\dim X_{\lambda}$).
\subsection{Proof of (iii)}
It follows from the discussion in the beginning of~Subsection~\ref{i} that

\medskip
(a) If an $\calO^{\x}$-orbit $X$ on $\Gr_{\SL(2)}$ carries a non-zero $(\calO^{\x},\calL)$-equivariant sheaf then $\dim X\geq k$;

(b) $\dim X_0=k$.

\medskip
\noindent
It follows from (b) that $\overline{X}_0\backslash X$ is a union of $\calO^{\x}$-orbits of dimension $<k$. Thus (a) implies that the natural morphism $\calF^0_!\to\calF^0_*$ is an isomorphism.

\section{Proof of Theorem \ref{thm}(2)}
\label{section 3}

In this Section we prove the 2nd assertion of Theorem \ref{thm}. It is in fact a mild variation on the proof of the derived geometric Satake equivalence (cf.~\cite{befi}).
\subsection{Reduction to $\SL(2)$}
We are supposed to study the derived category of $\calK^{\x}$-equivaraint $D$-modules on $\Gr_{\GL(2)}$. We claim that it is the same as the derived category of $\calO^{\x}$-equivariant $D$-modules on $\Gr_{\SL(2)}$ (here $\calO^{\x}$ is embedded into $\SL(2,\calK)$ via the identification of the standard Cartan subgroup of $\SL(2)$ with $\GG_m$). Indeed, we have $\calK^{\x}=\calO^{\x}\x \ZZ$. The last factor acts simply  transitively on the set of connected components of $\Gr_{\GL(2)}$ and the first factor preserves every connected component. Hence a $\calK^{\x}$-equivariant $D$-module on $\Gr_{\GL(2)}$ is the same as an $\calO^{\x}$-equivariant $D$-module on the connected component of 1, which is equal to $\Gr_{\SL(2)}$. The reader must be warned that the action of $\calO^{\x}$ on $\Gr_{\SL(2)}$ coming from our usual $\calK^{\x}$-action on
$\Gr_{\GL(2)}$ is not the same as the action coming from the Cartan torus of $\SL(2)$, but the latter is obtained from the former my means of the map $x\mapsto x^2$ which doesn't change the equivariant derived category.

For the remainder of this section we shall write $\Gr$ instead of $\Gr_{\SL(2)}$.
\subsection{Koszul duality}
We let $D_{\calO^{\x}}(\Gr)$ denote the corresponding equivariant derived category; since orbits of $\calO^{\x}$ on $\Gr$ are parametrized by discrete set, we can work with constructible sheaves instead of $D$-modules.

We let $D^b_{\calO^\x}(\Gr)$ denote the bounded derived category of
$\calO^\x$-equivariant constructible sheaves on $\Gr$. Recall that we need to show the following:
\begin{thm}
$D^b_{\calO^\x}(\Gr)\simeq \on{Coh}((\VV\x \VV^*[2])/\GL(\VV))$.
\end{thm}

\subsection{Equivariant cohomology}
Let $\lam\in \ZZ_+,\mu\in \ZZ$. We are going to think about $\lam$ as a dominant coweight of $\PGL(2)$ and about $\mu$ as an arbitrary coweight of $\PGL(2)$. Let us also assume that $\lam-\mu\in 2\ZZ$. Then we define $\calF^{\lam,\mu}$ to be the IC-sheaf of $z^{\mu}\oGr^{\lam}$ (note that because $\lam$ and $\mu$ have the same parity, it follows that $z^{\mu}\oGr^{\lam}\subset \Gr_{\SL(2)}=\Gr$). This is an object of $D^b_{\calO^\x}(\Gr)$.  We would like to describe
$H^*_{\calO^\x}(\Gr,\calF^{\lam,\mu})$ as a module over $H^*_{\calO^{\x}}(\Gr,\CC)$.

First, let us describe $H^*_{\calO^{\x}}(\Gr,\CC)$. Namely, let $\bf{Det}$ denote the standard determinant line bundle on $\Gr$. Then we have
$$
H^*_{\calO^{\x}}(\Gr,\CC)=\CC[\bfa,\bfc]
$$
where $\bfa$ is the standard generator of $H^*_{\calO^\x}(\pt)=H^*_{\CC^\x}(\pt)$ and $\bfc=c_1(\bf{Det})$ (equivariant first Chern class).

We can now describe $H^*_{\calO^\x}(\Gr,\calF^{\lam,\mu})$.

\begin{prop}\label{c-sl}
Let $V(\lam)$ denote the irreducible representation of $\SL(2)$ with highest weight $\lam$ (it has dimension $\lam+1$). Let $\pi_{\lam}\colon \mathfrak{sl}_2\to \End(V(\lam))$ denote the corresponding map. Then
the $H^*_{\calO^{\x}}(\Gr,\CC)=\CC[\bfa,\bfc]$-module $H^*_{\calO^\x}(\Gr,\calF^{\lam,\mu})$ is isomorphic to
$\CC[\bfa]\otimes V(\lam)$ where

a) $\bfc$ acts by
\begin{equation}\label{princ}
\pi_{\lam}
\begin{pmatrix}
0 & 1\\
\bfa^2 & 0
\end{pmatrix}
+ \mu \bfa.
\end{equation}

b) The grading on $\CC[\bfa]\otimes V(\lam)$ is equal to the tensor product of the standard grading on $\CC[\bfa]$ (recall that $\bfa$ has degree 2) and the grading on $V(\lam)$ by eigenvalues of $h$ (here we use the standard basis $(e,h,f)$ of the Lie algebra of $\SL(2)$). Note that the endomorphism of $\CC[\bfa]\otimes V(\lam)$ given by the  element \ref{princ} is homogeneous of degree 2 with respect to this grading.
\end{prop}
\begin{proof}
This statement is well-known when $\mu=0$. To prove it for general $\mu$ it is enough to show that
$c_1((z^{\mu})^*\mathbf{Det})=\bfc+\mu \bfa$. It is enough to check this equality after restricting to every $\calO^{\x}$-fixed point on $\Gr$ where it is obvious.
\end{proof}
Let us slightly reformulate this answer. Given $\lam$ and $\mu$ as above let $V(\lam,\mu)$ denote the (unique) irreducible representation of $\GL(2)$, such that its restriction to $\SL(2)$ is isomorphic to $V(\lam)$ and its central character is given by $\mu$ (note that such a representation exists precisely when $\lam-\mu\in 2\ZZ$). In what follows we shall regard it as a graded vector space, where the grading as before is given by the eigenvalues of $h\in \mathfrak{sl}_2$. Let $\pi_{\lam,\mu}\colon \mathfrak{gl}_2\to \End(V(\lam,\mu))$ denote the corresponding map.
Then~(\ref{princ}) is equal to
\begin{equation}\label{princ-gl}
\pi_{\lam,\mu}
\begin{pmatrix}
\bfa & 1\\
\bfa^2 & \bfa
\end{pmatrix}.
\end{equation}
Let us make yet another reformulation of the answer.
Let
$$
S(\bfa)=\begin{pmatrix}
\bfa & 1\\
\bfa^2 & \bfa
\end{pmatrix},\quad
T(\bfa)=
\begin{pmatrix}
0 & 1\\
0 & 2\bfa
\end{pmatrix}
$$
Then $T(\bfa)=g(\bfa)^{-1} S(\bfa) g(\bfa)$
where
$$
g(\bfa)=
\begin{pmatrix}
1 &  0 \\
-\bfa & 1
\end{pmatrix}.
$$

Hence we get the following equivalent version of Proposition \ref{c-sl}:
\begin{prop}\label{c-gl}
The $H^*_{\calO^{\x}}(\Gr,\CC)=\CC[\bfa,\bfc]$-module $H^*_{\calO^\x}(\Gr,\calF^{\lam,\mu})$ is isomorphic to
$\CC[\bfa]\otimes V(\lam,\mu)$ where $\bfc$ acts by $\pi_{\lam,\mu}(T(\bfa))$.
\end{prop}

\subsection{The functor}
We can now describe the functor $F\colon D^b_{\calO^\x}(\Gr)\to\on{Coh}((\VV\x\VV^*[2])/\GL(2))$.
Namely, it has the property that
$$
F(\calF^{\lam,\mu})=\calO_{\VV\x\VV^*[2]}\otimes V(\lam,\mu)
$$
where the group $\GL(2)$ acts on the RHS diagonally. We claim that in order to check existence of $F$ it is enough
to construct isomorphisms
\begin{equation}\label{isom}
\begin{aligned}
&\Ext_{D^b_{\calO^{\x}}(\Gr)}(\calF^{\lam,\mu},\calF^{\lam',\mu'})\simeq \\ & \Ext_{\bfO_{\VV\x\VV^*[2]}\rtimes \GL(2)}(\bfO_{\VV\x\VV^*[2]}\otimes V(\lam,\mu),\bfO_{\VV\x\VV^*[2]}\otimes V(\lam',\mu'))
\end{aligned}
\end{equation}
for any $(\lam,\mu)$ and $(\lam',\mu')$ as above (these isomorphisms must be compatible with compositions). Indeed,
if we have such isomorphisms then a word-by-word repetition of the arguments of~\cite[Section~6]{befi} constructs the functor $F$ (and also proves that it is an equivalence).

\subsection{Computing Ext's}
The next result allows us to compute Ext's between $\calO^{\x}$-equivariant IC-sheaves on $\Gr$; it is analogous to a
theorem of V.~Ginzburg from \cite{ginz} but we do not know how to prove it by any general argument.
\begin{prop}\label{ginz}
\begin{equation}\label{ginzburg}
\Ext_{D^b_{\calO^{\x}}(\Gr)}(\calF^{\lam,\mu},\calF^{\lam',\mu'})=\Hom_{H^*_{\calO^{\x}}(\Gr,\CC)}(H^*_{\calO^{\x}}(\Gr,\calF^{\lam,\mu}), H^*_{\calO^{\x}}(\Gr,\calF^{\lam',\mu'})).
\end{equation}
Here we use the following convention: when we write Hom between two graded modules over a graded ring we consider {\em all} homomorphisms (not just those that preserve the grading).
\end{prop}
\begin{proof}
Obviously, we have a map from the LHS of (\ref{ginzburg}) to the RHS of (\ref{ginzburg}). First, we claim that his map is injective.
For this it is enough to show the following:

(1) Both sides are free modules over $H^*_{\calO^{\x}}(\pt)$;

(2) The map in question becomes an isomorphism after tensoring with the field of fractions of $H^*_{\calO^{\x}}(\pt)$.

\noindent
The first assertion is known to follow from the fact that the corresponding non-equivariant Ext's  and cohomologies are pure (which follows from the fact that these are Ext's between pure sheaves on a projective variety). The second assertion follows from localization theorem since the set of fixed points of $\CC^{\x}\subset \calO^{\x}$ in the closure of any $\calO^{\x}$-orbit on $\Gr$ is finite.

Now let us show that the above map is surjective.
It follows from Proposition~\ref{c-gl} that $H^*_{\calO^{\x}}(\Gr,\calF^{\lam,\mu})$ is a cyclic $H^*_{\calO^{\x}}(\Gr,\CC)=\CC[\bfa,\bfc]$-module
generated by one vector $v_{\lam,\mu}$ of degree $-\lam$ whose annihilator is generated by the element
\begin{equation}\label{cyclic}
\prod\limits_{i=0}^{\lam}(\bfc-\bfa(2i+\mu-\lam)).
\end{equation}

Let now $(\lam,\mu)$ and $(\lam',\mu')$ be as in (\ref{ginzburg}). Let $S(\lam,\mu)$ be the set $\{\mu-\lam,\mu-\lam+2,\cdots,\lam\}$ (respectively, let $S(\lam',\mu')=\{\mu'-\lam',\mu-\lam+2,\cdots,\lam'\}$). Let $k$ be the cardinality of $S(\lam,\mu)\cap S(\lam',\mu')$.
Then the RHS of (\ref{ginzburg}) is

a) equal to $0$ if $k=0$;

b) generated by one element of degree $2(\lam'+1-k)$ whose annihilator in $\CC[\bfa,\bfc]$ is generated by
$\prod\limits_{i\in S(\lam,\mu)\cap S(\lam',\mu')} (\bfc-\bfa i)$ for $k>0$.

\noindent
We now want to compare this to the LHS of (\ref{ginzburg}). Let $\oGr^{\lam,\mu}$ denote the support of $\calF^{\lam,\mu}$. Since $\calF^{\lam,\mu}$ (resp. $\calF^{\lam',\mu'}$) is the constant sheaf on $\oGr^{\lam,\mu}$ (resp. on $\oGr^{\lam',\mu'}$) shifted by
$\lam$ (resp. by $\lam'$), it follows that the LHS of (\ref{ginzburg}) is equal to $H^*_{\calO^*}(\oGr^{\lam,\mu}\cap \oGr^{\lam',\mu'},\CC)[\lam'-\lam]$.
Thus, Proposition \ref{ginz} follows from the following:
\begin{lem}
\begin{enumerate}
\item
$\oGr^{\lam,\mu}\cap \oGr^{\lam',\mu'}=\emptyset$ if $k=0$.
\item
$\oGr^{\lam,\mu}\cap \oGr^{\lam',\mu'}=\oGr^{\lam'',\mu''}$,
where $\lam'',\mu''$ are such that $S(\lam,\mu)\cap S(\lam',\mu')=S(\lam'',\mu'')$ (for $k>0$).
\end{enumerate}
\end{lem}
\begin{proof}
Thre assignment $\mu\mapsto z^{\mu}$ defines a bijection between $2\ZZ$ and $\Gr^{\CC^{\x}}$. Any closed $\calO^*$-invariant subset of $\Gr$ is uniquely determined by its intersection with $\Gr^{\CC^{\x}}=2\ZZ$. It is easy to see that $\oGr^{\lam,\mu}\cap\Gr^{\CC^{\x}}=S(\lam,\mu)$, hence the lemma follows.
\end{proof}
The proposition is proved. \end{proof}

\subsection{}
\label{311}
We need to construct an isomorphism between the RHS of~(\ref{isom}) and the RHS
of~(\ref{ginzburg}). Note that the latter is equal to Hom between two explicit modules over the ring $\CC[\bfa,\bfc]$ over the polynomial ring in two variables of degree 2. We would like to rewrite the former in a similar way. For this let us do the following.

First, let $P$ denote the stabilzer of the vector $(1,0)$ in $\VV$. Then we claim that
$$
\begin{aligned}
\Ext_{\bfO_{\VV\x\VV^*[2]}\rtimes \GL(2)}(\bfO_{\VV\x\VV^*[2]}\otimes V(\lam,\mu),\bfO_{\VV\x\VV^*[2]}\otimes V(\lam',\mu'))=\\
\Hom_{\bfO_{\VV^*[2]}\rtimes P}(\bfO_{\VV^*[2]}\otimes V(\lam,\mu),\bfO_{\VV^*[2]}\otimes V(\lam',\mu')).
\end{aligned}
$$
Indeed, since we are computing Hom's between free modules, we can replace $\VV$ by $\VV\backslash \{ 0\}$. Since $\GL(2)$ acts transitively on the latter with $P$ being the stabilizer of one element we obtain the above isomorphism.

Now we would like to describe a functor from the category of $P$-equivariant coherent sheaves on $\VV^*[2]$ to the category of graded modules over
$\CC[\bfa,\bfc]$ which is fully faithful on free modules. The category of $P$-equivariant coherent sheaves on $\VV^*[2]$ can be thought of as the category of $P$-equivariant graded modules over $\CC[x,y]$ where $x$ and $y$ both have degree 2. The group $P$ consists of matrices
\begin{equation}\label{P}
g=
\begin{pmatrix}
1  & \alpha\\
0 & \beta
\end{pmatrix}.
\end{equation}
Such a matrix acts on a vector $(x,y)$ to by means of $(g^t)^{-1}$ (here $g^t$ stands for the tranposed matrix).
Thus the Lie algebra of $P$ consists of matrices of the form
$$
A=
\begin{pmatrix}
0  & u\\
0 & v
\end{pmatrix}
$$
and $A(x,y)=(0,-ux-vy)$.

Let us take a module $M$ as above and let us restrict it to the line $y=-1$, i.e. consider the quotient $M/(y+1)M$.
This  quotient is endowed with a natural action of $\CC^{\x}$ which comes from the $\CC^{\x}$-action on $M$ coming from the grading on $M$ and the action coming from the embedding $\CC^{\x}\hookrightarrow P$ corresponding to matrices as in~(\ref{P}) with $\alpha=0$.
We would like to extend this to a structure of a graded $\CC[\bfa,\bfc]$-module on it.

The action of $\bfa$ just comes from the action of $x/2$ on $M$. The action of $\bfc$ is characterized by the property that its action on the fiber over the point $(x,-1)=(2\bfa,-1)$ is
given by the action of the matrix
\begin{equation}\label{cent}
\begin{pmatrix}
0 & 1\\
0 & 2\bfa
\end{pmatrix}\in \on{Lie}(P).
\end{equation}
This makes sense because this matrix kills the vector $(2\bfa,-1)$ and hence the corresponding one-parametric subgroup (and hence also its Lie algebra) acts on the fiber of any $P$-equivariant coherent sheaf over $(2\bfa,-1)$.

Let us denote the resulting functor from $P$-equivariant coherent sheaves on $\VV^*[2]$ to graded $\CC[\bfa,\bfc]$-modules by $\tilF$.
It follows from Proposition \ref{c-gl} that this functor sends the module $\bfO_{\VV^*[2]}\otimes V(\lam,\mu)$ to $H^*_{\calO^{\x}}(\Gr,\calF^{\lam,\mu})$.
To finish the proof it remains to show that $\tilF$ is fully-fathful on free modules. This immediately follows from the following two (easy) statements:

(1) $P\cdot \{(x,-1)\}=\VV^*\backslash \{ 0\}$;

(2) The stabilizer of the point $(2\bfa,-1)$ in $P$ is equal to the one-parametric subgroup generated by the matrix~(\ref{cent}).

\subsection{Abelian equivalence}
We would like to conclude this Section with a variant of Theorem \ref{thm}(2) which in particular will give rise to certain equivalence of abelian categories (this is not strictly speaking needed for the purposes of this paper, but it is important for some future work). Namely, first of all
we claim that the category $\Coh( (\VV\x \VV[-1])/\GL(\VV))$ is equivalent to the derived category of $\GL(\VV)$-equivariant finitely generated modules over
the algebra $\Lam(\VV)\otimes \Lam(\VV^*)$. Indeed, $\Coh( (\VV\x \VV[-1])/\GL(\VV))$ is the derived category of $\GL(\VV)$-equivariant dg-modules over $\Sym(\VV^*)\otimes \Lam(\VV^*[-1])$ (considered as a dg-algebra with trivial differential).\footnote{Here when we write $\Lam(W[d])$ (for a vector space $W$ and $d\in \ZZ$), we just mean the dg-algebra with trivial differential which is equal to the exterior algebra generated by elements of $W$ which have homological degree $-d$, i.e.\ we are NOT using the ``super-notation" here with respect to the homological degree. Same goes for the notation $\Sym(W[d])$.}

Let now $M$ be any $\GL(\VV)$-equivariant
dg-module over $\Sym(\VV^*)\otimes \Lam(\VV^*[-1])$.  Define a new grading of $M$ which is equal to the sum of the old grading and the grading coming from the action of the center of $\GL(\VV)$. This makes it into a $\GL(\VV)$-equivariant dg-module over $\Sym(\VV^*[1])\otimes \Lam(\VV^*)$. By applying Koszul duality with respect to the first factor we can now associate to $M$ a finitely generated $\GL(\VV)$-eqvuivariant module
 over $\Lam(\VV)\otimes \Lam(\VV^*)$. It is easy to see  that this procedure defines an equivalence between the derived category of $\GL(\VV)$-equivariant dg-modules over $\Sym(\VV^*)\otimes \Lam(\VV^*[-1])$ and the derived category of $\GL(\VV)$-equivariant modules over $\Lam(\VV)\otimes \Lam(\VV^*)$. The advantage of the latter model is that it comes equipped with an obvious $t$-structure, whose heart is the abelian category of $\GL(\VV)$-equivariant modules over the algebra $\Lam(\VV)\otimes \Lam(\VV^*)$.

 On the other hand, the category $D^b_{\calK^{\x}}(\Gr_{\GL(2)})$ also has an obvious $t$-structure whose heart can be identified with the category $\Perv_{\calK^{\x}}(\Gr_{\GL(2)})$ of $\calK^{\x}$-equivariant perverse sheaves on $\Gr_{\GL(2)}$ (the latter category is the same as
 $\Perv_{\calO^{\x}}(\Gr_{\SL(2)})$ which is just the full subcategory of the category of perverse sheaves (with finite-dimensional support) on
 $\Gr_{\SL(2)}$ which are constant along $\calO^{\x}$-orbits).

 The following statement is an easy corollary of the proof of Theorem \ref{thm}(2); we leave the details to the reader.
\begin{thm}\label{abelian}
The equivalence between $D^b_{\calK^{\x}}(\Gr_{\GL(2)})$ and the derived category of $\GL(\VV)$-equivariant finitely generated modules over $\Lam(\VV)\otimes \Lam(\VV^*)$ (obtained by combining Theorem \ref{thm}(2) and the equivalence described in the beginning of this subsection) preserves the above $t$-structures. In particular, the category $\Perv_{\calK^{\x}}(\Gr_{\GL(2)})$ is equivalent to the abelian category of $\GL(\VV)$-equivariant finitely generated modules over the algebra $\Lam(\VV)\otimes \Lam(\VV^*)$.
\end{thm}

\section{Proof of Theorem \ref{thm}(3)}

\subsection{Compact objects in $\Dmod_\bfH(X)$}
Let $X$ be a scheme of finite type over $\CC$. Let also $\bfH$ be a pro-algebraic group over $\CC$ acting on $X$; we assume that $\bfH$ has a normal pro-unipotent subgroup with finite dimemensional quotient. As before, we denote by $\Dmod_\bfH(X)$ the derived category of strongly $\bfH$-equivariant $D$-modules on $X$. We also denote by $D^b_\bfH(X)$ its full subcategory consisting of bounded complexes with coherent cohomology. We would like to get a characterization of compact objects in $\Dmod_\bfH(X)$ (under some additional assumptions).
This question is studied in detail in \cite{DrGa}. The following lemma is an easy consequence of the results of {\em loc.\ cit.}:
\begin{lem}\label{compact}
\begin{enumerate}
\item
Assume that $\calF\in \Dmod_\bfH(X)$ is compact. Then $\calF\in D^b_\bfH(X)$.
\item
Assume that $\calF\in \Dmod_\bfH(X)$ is compact. Then its equivariant de Rham cohomology $H^*_{\bfH}(X,\calF)$ is finite-dimensional (i.e.\ it is a bounded complex of vector spaces with finite-dimensional cohomology).

\item
Assume that $X=\pt$. Then conditions (1) and (2) above are also sufficient for compactness.
\item
Let $\bfH=\CC^{\x}\times \bfH^0$ where $\bfH^0$ is (pro)unipotent.
Then $\calF\in D^b_{\bfH}(X)$ is compact if and only if for any embedding
$i_x\colon\{x\}\to X$ of $\CC^{\x}$-fixed point $x$ in $X$ the object $i_x^!\calF$ is a compact object of $\Dmod_{\CC^{\x}}(\pt)$.
\end{enumerate}
\end{lem}

\subsection{The cohomology functor}
In view of assertion (2) of Lemma \ref{compact} we would like to describe what happens to the functor of equivariant de Rham cohomology under the equivalence constructed
in~Section~\ref{section 3}. Let us denote this equivalence by $\Phi$ (this is a functor from $D^b_{\calO^{\x}}(\Gr)$ to $\Coh((\VV\times \VV^*[2])/\GL(\VV)$).

Let us consider the closed dg-subscheme $\SS$ of $\VV\times \VV^*[2]$ consisting of pairs $(v,v^*)$
where $v=(1,0)$ and $v^*$ is of the form $(x,-1)$. Then we claim the following
\begin{lem}
We have canonical isomorphism
\begin{equation}\label{cohom}
H^*_{\calO^{\x}}(\Gr,\calF)\simeq \calF|_{\SS}
\end{equation}
for any $\calF\in D^b_{\calO^{\x}}(\Gr)$.
Here the grading on the RHS of (\ref{cohom}) is defined in the same way as in~Section~\ref{311}.
\end{lem}
The proof follows immediately from the construction of the functor $\Phi$ described
in~Section~\ref{section 3}.
\subsection{Compact objects in $\calD_{\calO^{\x}}(\Gr)$}
Let us now go back to the proof of~Theorem~\ref{thm}(3).
We want to show that an object $\calF$ in $\calD_{\calO^{\x}}(\Gr)$ is compact if and only if it is a bounded complex of coherent $D$-modules (which in this case is the same as a bounded complex of constructible sheaves)
and $\Phi(\calF)$ is supported on $\calZ_\VV$. Let us first show the ``only if" direction.
According to assertion (2) of Lemma \ref{compact} compactness of $\calF$ implies that $H^*_{\calO^{\x}}(\Gr,\calF)$ is finite-dimensional. This condition is equivalent to the condition $\dim \on{supp}(\Phi(\calF))\cap \SS=0$; here we regard both $\on{supp}(\Phi(\calF))$ and $\SS$ as closed subvarieties of $\VV\times \VV^*$ (i.e.\ we disregard the cohomological grading on the 2nd factor). However, the fact that $\Phi(\calF)$ is actually an object of $\Coh((\VV\times \VV^*[2])/\GL(\VV))$  implies that
$\supp(\Phi(\calF))$ is

(a) $\GL(\VV)$-invariant.

(b) $\CC^{\x}$-invariant where the $\CC^{\x}$-action on $\VV\times \VV^*$ comes from dilating the 2nd factor.

\noindent
It is easy to see that a closed subvariety of $\VV\times \VV^*$ which satisfies conditions (a) and (b) above has zero-dimensional intersection with $\SS$ if and only if it is contained in $\calZ_{\VV}$, which finishes the proof of the ``only if" direction.

\subsection{End of the proof}
To prove the ``if" direction we are going to use the 4th assertion of~Lemma~\ref{compact} (note that $\calO^{\x}$ is a product of $\CC^{\x}$ and a pro-unipotent group).
Let us assume that $\on{supp}(\Phi(\calF))\subset \calZ_{\VV}$.
Combining the 3rd and 4th assertions we see that (using the notation of~Section~\ref{section 3})
we just need to check that for any even integer $\mu$ we have
\begin{equation}\label{ineq}
\dim \Ext^*(\calF^{0,\mu},\calF)<\infty
\end{equation}
(here we compute Ext in the equivariant derived category).
Indeed, the sheaves $\calF^{0,\mu}$ are exactly the sky-scraper sheaves at the $\CC^{\x}$-fixed points in $\Gr$.

First of all, we claim that it is enough to assume that $\mu=0$. Indeed, we have
$$
\Ext^*(\calF^{0,\mu},\calF)=\Ext^*(\calF^{0,0},(z^{-\mu})^*\calF)
$$
and $\Phi((z^{-\mu})^*\calF)=\Phi(\calF)\otimes V(0,-\mu)$, hence if $\Phi(\calF)$ is suported inside $\calZ_{\VV}$ then the same is true for $\Phi((z^{-\mu})^*\calF))$.

Now, since $\Phi(\calF^{0,0})=\bfO_{\VV\x \VV^*[2]}$ it follows that
$$
\RHom(\calF^{0,0},\calF)=\Phi(\calF)^{\GL(\VV)}.
$$
To show that the RHS of the above equation has finite dimensional cohomology (assuming that
$\Phi(\calF)$ is supported inside $\calZ_{\VV}$) it is enough to show $\bfO_{\calZ_{\VV}}^{\GL(\VV)}$ is finite-dimensional (since $\Phi(\calF)$ is a finite extension of quotients of $\bfO_{\calZ_{\VV}}$). This immediately follows from the fact
that $\bfO_{\VV\times \VV^*}^{\GL(\VV)}=\CC[v^*(v)]$ which is obvious (here we regard $v^*(v)$ as a function $\VV\times \VV^*\to \CC$).

\end{document}